\documentclass[english]{elsarticle}
\usepackage[T1]{fontenc}
\usepackage[latin9]{inputenc}
\usepackage{color}
\usepackage{float}
\usepackage{amsmath}
\usepackage{amssymb}
\usepackage{graphicx}
\usepackage{fullpage}
\usepackage{lmodern}
\usepackage{url}
\usepackage{mathtools,amssymb}
\usepackage{algorithm}
\usepackage{amsthm}
\usepackage{hyperref}

\newtheorem{theorem}{Theorem}

\newtheorem{lemma}{Lemma}


\usepackage{babel}
\begin{document}
\title{A robust and conservative dynamical low-rank algorithm}
\author[uibk]{Lukas Einkemmer} \ead{lukas.einkemmer@uibk.ac.at}
\author[uibk]{Alexander Ostermann} \ead{alexander.ostermann@uibk.ac.at}
\author[univaq]{Carmela Scalone\corref{cor1}} \ead{carmela.scalone@univaq.it}
\address[uibk]{Department of Mathematics, University of Innsbruck, Austria}
\address[univaq]{Dipartimento di Ingegneria e Scienze dell'Informazione e Matematica, University of L'Aquila, Italy}
\cortext[cor1]{Corresponding author}
\begin{abstract}
Dynamical low-rank approximation, as has been demonstrated recently, can be extremely efficient in solving kinetic equations. However, a major deficiency is that it does not preserve the structure of the underlying physical problem. For example, the classic dynamical low-rank methods violate mass, momentum, and energy conservation. In [L. Einkemmer, I. Joseph, J. Comput. Phys. 443:110495, 2021] a conservative dynamical low-rank approach has been proposed. However, directly integrating the resulting equations of motion, similar to the classic dynamical low-rank approach, results in an ill-posed scheme. In this work we propose a robust, i.e.~well-posed, low-rank integrator that conserves mass and momentum (up to machine precision) and significantly improves energy conservation. We also report improved qualitative results for some problems and show how the approach can be combined with a rank adaptive scheme.
 \end{abstract}
\maketitle

\section{Introduction}

Kinetic equations are used to model physical phenomena ranging from the transport of edge localized modes to the divertor in magnetic confinement fusion devices, to calculating dosage in radiation therapy. They provide a better picture of the underlying physics in situations where the assumptions made to derive fluid models (i.e.~that the system is in thermodynamic equilibrium) do not hold. However, due to the dependence of these partial differential equations on both physical space $x$ and velocity space $v$, the problem is posed in an up to six-dimensional phase space. This implies that the memory requirements and computational cost scale as $\mathcal{O}(n^6)$, where $n$ is the number of grid points per dimension. This unfavorable scaling is often referred to as the curse of dimensionality and implies that directly solving such problems is extremely expensive.

While particle methods have been historically used to alleviate this computational burden, it is also well known that they suffer from excessive noise which makes it difficult to resolve certain phenomena, such as Landau damping. Recently, dynamical low-rank approximations, as introduced in \cite{KL07}, have been used successfully for kinetic problems. This approach approximates the six-dimensional dependence of the density function by a set of lower dimensional equations, thus drastically reducing the required computational and memory cost.

It has been demonstrated that dynamical low-rank approximation works well for kinetic models in many cases (see, e.g., \cite{El18,einkemmer2020low,cassini2021efficient} for problems in plasma physics and \cite{peng2019low,einkemmer2020asymptotic,kusch2022low,kusch2021robust} for radiation transport), and even some progress has been made in analyzing such methods mathematically \cite{ding2019error,kusch2021stability}. The drawback of the classic dynamical low rank integrators is that the approximation made destroys the inherent physical structure of the problem. To be more precise, the important physical invariants of mass, momentum, and energy are no longer conserved (even before a space and time discretization is introduced). This can significantly impact the numerical solution and lead to qualitatively incorrect results.

In some work it has been suggested to restore mass conservation by renormalizing the distribution function; the obvious disadvantage being that this approach can not be extended to achieve simultaneous conservation of both mass and momentum. Moreover, it has been shown in \cite{EL18_cons} that just conserving the invariants is insufficient in order to obtain good qualitative results. In fact, for each of mass, momentum, and energy there is an underlying conservation law (the moment equation) that needs to be enforced in order to obtain physically reasonable results. The latter, in particular, has proved challenging. Both \cite{EL18_cons} and \cite{peng2020high} have used a correction to improve (but not remove) the error made in these conservation laws.

Recently, in \cite{EJ} a different method has been used to obtain the first dynamical low-rank integrator that is mass, momentum, and energy conservative from first principle. The fundamental idea is to modify the Galerkin condition in such a way that the equations of motion enforce the momentum equations, thus mimicking the structure of the original equation. This also implies conservation of mass, momentum, and energy. Combined with appropriate time and space discretization the approach then results in a fully conservative dynamical low-rank algorithm. This approach has since been used to obtain conservative low-rank algorithms in a different context \cite{guo2022conservative}.

In this paper we show that the approach described can be combined with an integrator that is robust in the presence of small singular values. Such robust integrators are desirable as they remove the need for regularizing the equations of motion, which can again destroy conservation. They also make it much easier to implement rank adaptive methods. We start with the unconventional integrator of Ceruti and Lubich described in \cite{ceruti2020unconventional}. We show how this integrator can be modified such that a fully conservative scheme can be attained. In particular, this requires us to add a set of basis functions to the approximation space in order that the moment equations are satisfied. This increases the rank and thus a truncation has to be performed after each time step. We can, however, perform this truncation in a conservative way. The resulting scheme conserves mass and momentum (as it preserves the corresponding conservation laws) up to machine precision, and it shows improved energy conservation. This, in practice, results in a method with much better properties compared to the original dynamical low-rank algorithm.

In this work we illustrate the algorithm for the Vlasov--Poisson model commonly used in plasma physics (see, section \ref{sec2} for a discussion of the model and its physical structure). \textcolor{black}{It should be noted, however, that since our algorithm relies on a way to enforce conservation by mimicking the continuous system, our approach can be easily applied to other kinetic equations as well}. Based on the description of the conservative dynamical low-rank integrator in section  \ref{sec3}, we outline our approach in section \ref{sec:robustcons}.
In section \ref{rank_adaptivity}, we present the rank adaptive approach. Finally, we show numerical results for linear and nonlinear Landau damping and the two-stream instability in section \ref{numerics}.

\section{Vlasov--Poisson equations and their physical structure}
\label{sec2}
Under the assumption of a constant background ion density, the time evolution of the electron density $f(t,x,v)$ in a collisionless plasma in the electrostatic regime is modeled by the Vlasov--Poisson equations
\begin{align*}
\partial_t f(t,x,v) + v \cdot \nabla_x f(t,x,v) -E(f)(t,x) \cdot \nabla_v f(t,x,v) = 0,\\
\nabla_x \cdot E(f)(t,x) =  1- \int f(t,x,v) \,dv, \quad \nabla_x \times E(f)(t,x)=0,
\end{align*}
where $(x,v) \in \Omega = \Omega_x \times \Omega_v$ with $\Omega_x \subset \mathbb{R}^d$ and $\Omega_v \subset \mathbb{R}^d$. The unknown is the density function $f(t,x,v)$. The electric field is usually expressed in terms of a potential $\phi$, i.e.~$E = - \nabla \phi$, which solves the following Poisson problem
\[ -\Delta \phi(f)(t,x) = 1 - \int f(t,x,v) \,dv. \]

This is a typical example of a physical phenomenon described by a kinetic model.
Since this model is posed in a $2d$-dimensional phase space (with $d\leq 3$), the numerical approximation of the solution on a grid is very expensive.

The Vlasov--Poisson equations have a rich physical structure. More specifically, they form a non-canonical Hamiltonian system with an infinite number of invariants (see, e.g.,~\cite{morrison1980maxwell, casas2017high, CEF15} for more details). Here we will focus on conserving the physically important invariants of mass $M$ and momentum $J$ given by
$$M(t)=\int_{\Omega} f(t,x,v) \,d(x,v), \quad \quad J(t)=\int_{\Omega} v f(t,x,v) \,d(x,v),$$
and energy $\mathcal{E}$
$$  \mathcal{E} = \frac{1}{2}\int_{\Omega} v^2 f(t,x,v) \,d(x,v)+ \frac{1}{2} \int_{\Omega_x}E(t,x)^2 \,dx.$$
\textcolor{black}{Note that in the multi-dimensional setting $v$ and $E$ are (up to) three-dimensional vectors and we use $v^2$ and $E^2$ to denote their squared Euclidean norm, e.g. $v^2 = \sum_i v_i^2$.}

For each of these invariants an associated conservation law (often called a moment equation or continuity equation) is satisfied. For mass it relates the mass density (or charge density) $\rho(t,x)$ with the momentum density (or current density) $j(t,x)$ as follows
$$ \partial_t \rho + \nabla_x \cdot j=0, \qquad \rho(t,x)=\int_{\Omega_v} f(t,x,v) \,dv, \qquad j(t,x) = \int_{\Omega_v} v f(t,x,v) \,dv. $$
We note that this equation is posed in physical space only. Integrating the continuity equation with respect to $x$, we immediately obtain conservation of mass, i.e.~$M(t)=\text{const}$.
\textcolor{black}{
The momentum density satisfies the following continuity equation
$$ \partial_t j + \nabla_x \cdot \sigma(t,x)= -E(t,x)\rho(t,x), \qquad j(t,x)=\int_{\Omega_v} v f(t,x,v) \,dv, \qquad \sigma(t,x)= \int_{\Omega_v} (v \otimes v) f(t,x,v) \,dv.$$
The conservation of the momentum is obtained by recognizing that $ E(1 -\rho) = \nabla \cdot (E \otimes E - \frac{1}{2} E^2)$ and integrating in physical space.
Due to the normalization of the particle-density function we have $ \int E dx = 0$.}\\
For the energy density
\begin{equation} \label{e_expr}
e(t,x)= \frac{1}{2} \int_{\Omega_v} v^2 f(t,x,v) \,dv + \frac{1}{2} E(t,x)^2 
\end{equation}
we have that the following continuity equation
$$ \partial_t e(t,x) + \nabla_x \cdot Q(t,x) = E(t,x) \cdot (\partial_t E(t,x) -j(t,x)), \quad Q(t,x) = \frac{1}{2} \int_{\Omega_v}v v^2 f(t,x,v) \,dv$$
is satisfied. From the electrostatic version of Amp\`ere's law $\partial_t E(t,x) =j(t,x)$, the global conservation of the energy $\mathcal{E}$ is then obtained as follows
$$\partial_t \mathcal{E} = \partial_t \int e(t,x) dx = \int E(t,x) \cdot (\partial_t E(t,x) -j(t,x)) \,dv = 0.$$
In a similar way integrating in space implies conservation of momentum. For more details we refer the reader to \cite{EJ}.

\section{Conservative dynamical low-rank scheme}
\label{sec3}
This section is dedicated to providing a summary of the conservative dynamical low-rank approach introduced in \cite{EJ}, which is the starting point of our work. The density function $f(t,x,v)$ is represented in the following form
\begin{equation*}
f(t,x,v)=f_{0v}(v) \sum_{i,j=1}^r X_i(t,x) S_{ij}(t) V_j(t,v),
\end{equation*}
where $r$ is the rank of the approximation and $f_{0v}$ is a weight function. For the low-rank factors, we assume that $S \in \mathbb{R}^{r \times r}$, $X_i(t, \cdot) \in L^2(\Omega_x)$, and $V_j(t, \cdot) \in L^2(\Omega_v, f_{0v})$, i.e. to lie in the $L^2$ space weighted by $f_{0v}$, with corresponding (weighted) inner products defined as
$$ \langle X_i(t,\cdot), X_j(t,\cdot) \rangle_x = \int_{\Omega_x} X_i(t,x) X_j(t,x) \,dx$$
and
$$ \langle V_i(t,\cdot), V_j(t,\cdot) \rangle_v = \int_{\Omega_v} f_{0v}(v) V_i(t,v) V_j(t,v) \,dv.$$
\textcolor{black}{
Moreover, for $w\in L^2(\Omega_v, f_{0v})$, we use the notation
$$
\|w\| = \sqrt{\langle w,w\rangle_v}.
$$}
In order to obtain a mass, momentum, and energy conservative scheme we require that $1$, each component of $v$, and $v^2$ lie in $L^2(\Omega_v, f_{0v})$. 
Thus, any choice of $f_{0v}$ that decays sufficiently fast as $\Vert v \Vert_{\mathbb R^d} \to \infty$ is sufficient.

The fundamental observation in \cite{EJ} is that if $1$, $v$, and $v^2$ lie in the approximation space then the dynamical low-rank integrator mimicks the continuous system and the proof that shows mass, momentum, and energy conservation for the Vlasov--Poisson equations can be easily applied directly to the low-rank approximation as well. To accomplish this we fix a (small) number, say $m$, of basis functions $V_j$. The basis functions that will remain unchanged by the low-rank integrator are denoted as follows
$$U_a(v)=V_a(t,v), \quad  1 \leq a \leq m.$$
The remaining functions are denoted by
$$ W_p(t,v)=V_p(t,v), \quad m < p\le r$$
and will vary according to the low-rank integrator.

It must be emphasized that orthogonality conditions are required between all of those function (both fixed and varying), i.e. we must have
$$ \langle U_a, U_b \rangle_v = \delta_{ab},\quad 1\le a,b \le m \qquad \text{and} \qquad \langle W_p,V_j \rangle_{\textcolor{black}{v}} = \delta_{pj}, \quad m<p\le r, \ 1\le j \le r.$$
The choice of $f_{0v} = \exp{(-v^2 /2)}$, $m = 3$, $U_1(v) = 1 / \Vert 1 \Vert $, $U_2(v) = v / \Vert v \Vert $ and $U_3(v) = (v^2-1) / \Vert v^2-1 \Vert $ fulfills the constraints and guarantees mass, momentum and energy preservation in $1+1$ dimension. For three-dimensional problems there are three momenta that are conserved, i.e.~we set $U_1(v) = 1 / \Vert 1 \Vert $, $U_2(v) = v_1 / \Vert v_1 \Vert $, $U_3(v) =  v_2 /\Vert v_2\Vert$, $U_4(v) = v_3 / \Vert v_3 \Vert$, and $U_5(v) = (v^2-1) / \Vert v^2-1 \Vert$ and thus $m=5$.

The main issue then is to derive equations of motion that keep the first $m$ basis functions fixed, while respecting the orthogonality constraints. The corresponding Galerkin condition is given in \cite{EJ} and results in the following equations of motion
\begin{align}
\sum_i \dot{X}_i S_{ik} &= \left( V_k, \text{RHS}\right)_v -\sum_i {X}_i \dot{S}_{ik} \label{eq:x}, \qquad 1 \leq k \leq r\\
\sum_{ip} S_{iq} S_{ip} \dot{W}_p&=\frac{1}{f_{0v}} \sum_i S_{iq} \textcolor{black}{\langle} X_i, \text{RHS} \textcolor{black}{\rangle}_x - \sum_{il} S_{iq} \dot{S}_{il} V_l,\qquad m+1 \leq q \leq r \label{eq:w}\\
\dot{S}_{kl} &= \left( X_kV_l, \text{RHS}\right)_{xv}, \qquad 1 \leq k,l \leq r, \label{eq:s}
\end{align}
where 
\[ 
    \textcolor{black}{\langle} f, g \textcolor{black}{\rangle}_x = \int_{\Omega_x} f g \,dx, \qquad (f, g)_v = \int_{\Omega_v} f g \,dv, \qquad (f, g)_{xv} = \int_{\Omega} f g \,d(x,v).
\]
Plugging in $\text{RHS}=-v \cdot \nabla_x f+E \cdot \nabla_v f$ we obtain
\begin{align*}
\sum_i \dot{X}_i S_{ik} &= - \sum_{ij} c^1_{kj} \cdot \nabla_x X_i S_{ij} + \sum_{ij} (c^2_{kj} \cdot E)  X_i S_{ij}-\sum_i {X}_i \dot{S}_{ik},\\
\sum_{ip} S_{iq} S_{ip} \dot{W}_p&=
\sum_i S_{iq} \left(\sum_{kl} d_{ik}^1[E] \cdot [\nabla_v(S_{kl} V_l) + \nabla_v(\log f_{0v}) S_{kl}V_l]  - \sum_{kl}(v \cdot d_{ik}^2)S_{kl} V_l \right)
 - \sum_{il} S_{iq} \dot{S}_{il} V_l, \\
\dot{S}_{kl}&=-\sum_{ij} (d_{ki}^2 \cdot c_{lj}^1) S_{ij} + \sum_{ij} (d_{ki}^1[E] \cdot c_{lj}^2) S_{ij},
\end{align*}
where the following coefficients are defined by taking the inner products componentwise
$$ c^1_{kj} = \langle V_k, v V_j \rangle_v, \quad \quad  c^2_{kj} = \left( V_k, \nabla_v(f_{0v} V_j) \right)_v $$
and
$$ d_{ik}^1[E] = \langle X_i, E X_k  \rangle_x, \quad \quad
d_{ik}^2 = \langle X_i, \nabla_x X_k \rangle_x. $$

The resulting equations are mass, momentum, and energy conservative. In order to turn this into a viable numerical scheme, a suitable time and space discretization has to be employed. With respect to space discretization this is relatively straightforward. The only requirement is that a discrete version of integration by parts holds true. This is the case for a range of methods. For example, centered finite differences, Fourier spectral methods, or discontinuous Galerkin schemes with centered flux satisfy this requirement. For the time integration the explicit conservative Euler scheme described in \cite{EJ} conserves mass and momentum up to machine precision (but not energy). In a similar way implicit schemes that conserve all three invariants can be obtained. The resulting numerical scheme satisfies discrete moment equations that once integrated (or rather summed) in physical space imply the global invariant.  For more details we refer the reader to \cite{EJ}.

The disadvantage of this approach, however, is that in order to solve for $X_i$ and $W_i$ in equations \eqref{eq:x}-\eqref{eq:w} we have to invert $S$. For increasingly accurate approximations this becomes increasingly ill-conditioned as $S$ contains small singular values. Thus, we do not obtain a robust integrator in the sense that choosing a rank that is too large compared to what is required for the numerical simulation can cause issues. We further note that since the conservative dynamical low-rank algorithm uses a different set of equations of motions it turns out to be incompatible with the projector splitting integrator proposed in \cite{LO14}, which in fact was the first robust dynamical low-rank integrator that has been constructed.

\section{The robust and conservative dynamical low-rank algorithm} \label{sec:robustcons}

Our goal now is to construct a robust and conservative dynamical low-rank algorithm. We start with the unconventional integrator that has recently been introduced in \cite{ceruti2020unconventional}. The main idea is to discretize equations \eqref{eq:x} and \eqref{eq:w} as follows:
\begin{align}
K_k^{n+1} &= K_k^n+\tau \left(V_k^n, \text{RHS}^n \right)_v, \qquad 1 \leq k \leq r,\label{eq:unconv_x} \\
L_q^{n+1} &=L_q^n+\frac{\tau}{f_{0v}} \sum_i S^n_{iq} \textcolor{black}{\langle} X_i^n,\text{RHS}^n \textcolor{black}{\rangle} _x-\tau \sum_{il}S_{iq}^n \left(X_i^n V_l^n, \text{RHS}^n \right)_{\textcolor{black}{xv}} V_l^n, \qquad m+1 \leq q \leq r\label{eq:unconv_w}
\end{align}
where
$$ \text{RHS} = -v \cdot \nabla_x f^n + E^n \cdot \nabla_v f^n, \qquad \qquad K_k^n=\sum_i X_i^n S_{ik}^n, \qquad \text{and} \quad L^n_q=\sum_{ip} S^n_{iq}S^n_{ip}W_p^{n}$$
and $\tau$ is the time step size.

We then perform a QR decomposition to
 $$K_k^{n+1} =  \sum_i X_i^{n+1} R_{ik}^1$$ and $$L_q^{n+1}= \sum_p W_p^{n+1} R_{pq}^2$$
 to obtain $X^{n+1}$ and $W^{n+1}$, discarding the $R$ parts, i.e.  $R_{ik}^1$ and $R_{pq}^2$, of the decomposition.

The main insight here is that the approximation spaces,
\[ \text{span}\{X_{i}^{n+1}\}_{i=1}^r, \qquad \text{span}\{ V_j^{n+1} \}_{j=1}^r = \text{span}\{ U_1, \ldots, U_m, W_{m+1}^{n+1}, \ldots, W_{r}^{n+1} \}, \] are uniquely defined by $K_k^{n+1}$  and $L_q^{n+1}$. The QR decomposition allows us to obtain a basis independent of how the singular values of $S$ look like, i.e.~in a robust way.

The final step is then to determine $S^{n+1}$ as follows
\textcolor{black}{
\begin{equation}
S_{kl}^{n+1}=\sum_{ij} M_{ki} S^n_{ij} N^{T}_{jl}+ \tau \left(X_k^{n+1} V_l^{n+1}, \text{RHS}
\left[f= f_{0v}\textstyle\sum_{kijl}X^{n+1}_k M_{ki}S^n_{ij} N^T_{jl} V^{n+1}_l\right] \right)_{xv},  \label{eq:unconv_S}
\end{equation}
}
where
\begin{equation}
\label{def_MN}
 M_{ki}=\langle  X_k^{n+1}, X_i^n\rangle_x, \quad N^T_{jl}=\langle  V_j^{n}, V_l^{n+1}\rangle_v.
 \end{equation}
\textcolor{black}{The notation $\text{RHS}\left[f= f_{0v}\textstyle\sum_{kijl}X^{n+1}_k M_{ki}S^n_{ij} N^T_{jl} V^{n+1}_l\right]$ means that we compute $\text{RHS}=-v \cdot \nabla_x f+E \cdot \nabla_v f$ by substituting $f= f_{0v}\textstyle\sum_{kijl}X^{n+1}_k M_{ki}S^n_{ij} N^T_{jl} V^{n+1}_l$}.

Note that the $M S^n N$ term in equation \eqref{eq:unconv_S} just expresses the solution at time $t^n$ using the already computed basis functions at time $t^{n+1}$. Specifically, the unconventional integrator makes the following approximation
\begin{equation}
    \sum_{ij} X_i^n S_{ij}^n V_j^n \approx \sum_{kijl} X_k^{n+1}M_{ki} S_{ij}^n N^T_{jl} V_l^{n+1}.
    \label{eq:MSN}
\end{equation}
\textcolor{black}{That is, we project the coefficient matrix $S^{n}_{ij}$ onto the updated basis spanned by $V^{n+1}_j$ and $X^{n+1}_i$ and then solve equation \eqref{eq:s}.}
This unconventional integrator for the dynamical low-rank method is robust with respect to the small singular values of $S$, since no matrix inversion needs to be performed.

However, even when applied to the equations of motion of the conservative dynamical low-rank integrator (as we have done here), it destroys conservation of mass, momentum, and energy. The reason for this is twofold:
\begin{enumerate}
    \item The projection of the subspaces $\text{span}\{X_i^n\}$ and $\text{span}\{V_j^n\}$ onto $\text{span}\{X_i^{n+1}\}$ and $\text{span}\{V_j^{n+1}\}$, which is performed in equation \eqref{eq:unconv_S}, does not conserve, in general, any of the invariants. That is, there is no guarantee that the left-hand side and the right-hand side of equation \eqref{eq:MSN} have, e.g., the same mass.
    \item The local conservation laws for mass is of the form $\partial_t \rho + \nabla \cdot j = 0$. While $j$ lies in $\text{span}\{X_i^n\}$ this is, in general, not true for $\nabla \cdot j$. In fact, the unconventional integrator approximates $\nabla \cdot j$ by a function in $\text{span}\{X_i^n\}$ and thus destroys the local conservation law. In a different setting this has already been observed in \cite{EL18_cons}.
\end{enumerate}

Fortunately, the unconventional integrator allows us to relatively easily augment the approximation space by additional basis functions. We consider the approximation spaces $\widetilde{X}^{n+1}$ and $\widetilde{V}^{n+1}$ as follows
\[ \widetilde{X}^{n+1} = (\widetilde{X}_1^{n+1}, \widetilde{X}_2^{n+1}, \dots),
    \qquad \widetilde{V}^{n+1} = (\widetilde{V}_1^{n+1}, \widetilde{V}_2^{n+1}, \dots), \]
where $(\widetilde{X}^{n+1}_j)_j$ is an orthonormal basis of $\text{span}\{ X_i^n, \nabla X_i^n, K_i^{n+1}\}$ and $(\widetilde{V}^{n+1}_j)_j$ is an orthonormal basis of $\text{span}\{ U, L_p^{n+1}, V_p^n\}$, respectively.

\textcolor{black}{Thanks to the augmentation of the basis in the proposed method, the following projections are exact. That is, recalling definition \eqref{def_MN}, we have
\begin{equation}
\label{proj_x}
P_{x,n+1}(X_j^n) \coloneqq \sum_i \langle  \widetilde{X}_i^{n+1}, X_j^n\rangle_x  \widetilde{X}_i^{n+1} = \sum_i M_{ij} \widetilde{X}_i^{n+1}  =   X_j^n,
\end{equation}
\begin{equation}
\label{proj_v}
P_{v,n+1}(V_l^{n}) \coloneqq \sum_m  \langle V_l^n, \widetilde{V}_m^{n+1}\rangle_v \widetilde{V}_m^{n+1} =
\sum_m N^T_{lm} \widetilde{V}_m^{n+1} =  V_l^{n},
\end{equation}
and
\begin{equation}
\label{proj_gx}
P_{x,n+1}(\nabla_x X_j^{n}) = \sum_k  \langle \widetilde{X}_k^{n+1}, \nabla_x X_j^n   \rangle_x  \widetilde{X}_k^{n+1} = \nabla_x X_j^{n}.
\end{equation}}

\begin{algorithm}
\caption{Robust and conservative dynamical low-rank integrator. \label{algorithm}}
\begin{enumerate}
    \item[]
    \item[] \hspace{-0.7cm}\textbf{Conservative unconventional dynamical low-rank integrator:}  \vspace{0.2cm}
\item Given $f^n = f_{0v} \sum_{ij} X_i^n S_{ij}^n V_j^n$, compute $K^{n+1}$ and $L^{n+1}$ from equations \eqref{eq:unconv_x} and \eqref{eq:unconv_w}, respectively.
\item 
 Compute $\widetilde{X}^{n+1} \in \mathbb{R}^{n_x \times 3r}$ by orthogonalization of the matrix $ \left[ X^n \;\; \nabla X^n \;\; K^{n+1}  \right]$ using a QR decomposition.
\item Compute $\widetilde{V}^{n+1}$ by orthogonalization of the matrix $\left[ U \;\;   L^{n+1}  \;\; W^n  \right]  $ using a QR decomposition. Note that $\widetilde{V}^{n+1}=\left[U \;\; \widetilde{W}^{n+1} \right] \in \mathbb{R}^{n_v {\times (2r-m)}}$, where $\widetilde{W}^{n+1} = [ \widetilde{V}^{n+1}_{m+1}, \ldots, \widetilde{V}_{r}^{n+1} ]$.
\item Compute $\widetilde{S}^{n+1}$ from equation \eqref{eq:unconv_S}.

\item[]
\item[] \hspace{-0.7cm}\textbf{Conservative truncation to rank $r$:} \vspace{0.2cm}
\item Compute $\widetilde{K} = \widetilde{X}^{n+1} \widetilde{S}^{n+1}$
and distribute it into two parts $$\widetilde{K} =[ \widetilde{K}^{cons} \; \widetilde{K}^{rem}],$$ where $\widetilde{K}^{cons}$ consists of the first $m$ columns of $\widetilde{K}$ and $\widetilde{K}^{rem}$ of the remaining columns.
\item Perform a QR decomposition of $\widetilde{K}^{cons}$, getting
$$ \widetilde{K}^{cons} = X^{cons} S^{cons}.$$
\item Perform a QR decomposition of $\widetilde{K}^{rem}$, getting
$$ \widetilde{K}^{rem} = \widetilde{X}^{rem} \widetilde{S}^{rem}.$$

\item Compute the singular value decomposition (SVD) of $\widetilde{S}^{rem}$, keep the largest $r-m$ singular values
$$ \widetilde{S}^{rem} \approx \hat{U}\hat{S} \hat{W}^T $$
and compute
$$X^{rem} = \widetilde{X}^{rem} \hat{U}, \quad S^{rem}= \hat{S}, \quad W^{n+1} =  \widetilde{W}^{n+1} \hat{W}.$$
\item Set $$V^{n+1} = \left[U\; W^{n+1} \right]. $$
\item Set $ \hat{X} = \left[ X^{cons} \; X^{rem} \right] $ and perform a QR decomposition $$\hat{X} = X^{n+1} R.$$

\item Set $$S^{n+1} = R \begin{bmatrix} S^{cons} & 0\\
0 & S^{rem}
\end{bmatrix}.$$

\item The computed approximation at time $t^{n+1}$ is then given by
$$f^{n+1} = \textcolor{black}{f_{0v}}\sum_{ij} X^{n+1}_i S^{n+1}_{ij} V^{n+1}_j.$$
\end{enumerate}
\end{algorithm}
 The resulting unconventional integrator then proceeds as stated in steps 1-4 of Algorithm \ref{algorithm}.
Note that this procedure increases the rank in every time step. Thus, we have to perform an appropriate truncation, which must be conservative and keeps the dominant singular values. We accomplish this by first projecting the obtained approximation on the $U_a$ in an exact way. The remainder is then treated by a singular value decomposition and only the largest $r-m$ singular values are kept.

\textcolor{black}{ In the following presentation we will rely on a notation where, after an appropriate discretization in space, the objects $X^{n}$, $K^n$, \ldots\ are considered as $n_x \times r$  matrices and the objects $V^{n}$ are considered as $n_v \times r$ matrices. This will significantly simplify the conservative truncation algorithm. Note, however, that in principle all operations used (such as rank, orthonormalization, QR, SVD) can also be directly applied to the functions ${X_i^n}$, etc.}

\textcolor{black}{In detail, we start by computing $\widetilde{K} = \widetilde{X}^{n+1} \widetilde{S}^{n+1}$
and distribute it into two parts:  $\widetilde{K}^{cons}$ consists of the largest $m$ columns of $\widetilde{K}$ and $\widetilde{K}^{rem}$ consists of the remaining columns.
 Performing a QR decomposition of $\widetilde{K}^{rem}$, we get 
$$ \widetilde{K}^{rem} = \widetilde{X}^{rem} \widetilde{S}^{rem}.$$
 Computing the singular value decomposition (SVD) of $\widetilde{S}^{rem}$, and keeping the first $r-m$ singular values, we obtain 
$$ \widetilde{S}^{rem} \approx \hat{U}\hat{S} \hat{W}^T. $$
We then set
$$X^{rem} = \widetilde{X}^{rem} \hat{U}, \quad S^{rem}= \hat{S}, \quad W^{n+1} =  \widetilde{W}^{n+1} \hat{W}.$$
 The new basis matrix $V^{n+1}$ is obtained by adding the fixed columns, i.e. $V^{n+1} = \left[U\; W^{n+1} \right]. $
Via a QR decomposition of $\widetilde{K}^{cons}$, we then compute 
$$ \widetilde{K}^{cons} = X^{cons} S^{cons}.$$
Finally, we set $ \hat{X} = \left[ X^{cons} \; X^{rem} \right] $ and compute $X^{n+1}$ via a QR decomposition $$\hat{X} = X^{n+1} R.$$  The final matrix of the coefficient $S^{n+1}$ is then defined as follows $$S^{n+1} = R \begin{bmatrix} S^{cons} & 0\\
0 & S^{rem}
\end{bmatrix}.$$}
That this procedure (see also steps 5-12 of Algorithm~\ref{algorithm}) is conservative can be seen as follows:
\begin{align*}
\langle  U_a, f^{n+1}   \rangle_v &= \langle U_a, \sum_{i,j} X_i^{n+1} S^{n+1}_{ij} V_j^{n+1} \rangle_v =  \sum_{i} X_i^{n+1} S^{n+1}_{ia} = \sum_{i} X_i^{n+1} \sum_j R_{ij}   \begin{bmatrix} S^{cons} & 0\\
0 & S^{rem}
\end{bmatrix}_{ja} \\
&= \sum_{j=1}^m \sum_i X_i^{n+1} R_{ij} S^{cons}_{ja} = \sum_{j=1}^m \hat{X}_j S_{ja}^{cons} = \sum_b X_b^{cons} S_{ba}^{cons} = \widetilde{K}_a.
\end{align*}

For notational simplicity we, from now on, will omit the tilde from $X^{n+1}$ and $V^{n+1}$. This is justified since we already know that the truncation step does not destroy conservation.

\subsection{Discrete conservation laws}

We now show that the entire algorithm is in fact mass and momentum conservative. 
\begin{theorem}
\label{thm_1}
 Algorithm \ref{algorithm}, for $m \geq 1$, satisfies the following discrete version of the continuity equation
    \begin{equation}
    \label{cont_d_rho}
     \frac{\rho^{n+1}-\rho^n}{\tau} + \nabla_x \cdot j^n = 0 
     \end{equation}
and thus preserves total mass.
\end{theorem}
\textcolor{black}{\begin{proof}
Since $V_1 = U_1 = \frac1{\|1\|}$, the mass density $\rho$ is given by
$$
\rho = (f,1)_v = \sum_{k,l} X_k S_{kl} \left\langle V_l,1\right\rangle_v = \sum_k X_i S_{k1} 
    \left\langle U_1,1\right\rangle_v,
$$
thus 
$$\rho^{n+1} =  \|1\| \sum_k X_i^{n+1} S_{k1}^{n+1}.$$
By replacing  $S_{k1}^{n+1}$ with the expression given in \eqref{eq:unconv_S}  and exploiting the properties of the augmented basis (equation \eqref{proj_x}), we get
\begin{align}
\label{rho_1}
 \rho ^{n+1} & =  \Vert 1 \Vert \sum_{k} X_k^{n+1} S_{k1}^{n+1} =  \Vert 1 \Vert \sum_{ij} S_{ij}^n \langle V_j^n, U_1 \rangle_v \sum_k X_k^{n+1} \langle X_k^{n+1}, X_i^{n}  \rangle_x 
      +\tau \Psi\\ \nonumber
&= \Vert 1 \Vert \sum_{i} S_{i1}^n  \sum_k X_k^{n+1} \langle X_k^{n+1}, X_i^{n}  \rangle_x  +\tau  \Psi\\ \nonumber 
&= \Vert 1 \Vert \sum_{i} S_{i1}^n   X_i^{n}  +\tau  \Psi \\ \nonumber 
&= \rho^n  +\tau  \Psi, \nonumber 
\end{align}
where 
\begin{equation}
\label{S1}
\Psi = \sum_k X_k^{n+1} \left(   X_k^{n+1},\text{RHS} \left[ f = f_{0v}\textstyle\sum_{ijlm} X_i^{n+1} M_{ij} S_{jl}^n N^T_{lm} V_m^{n+1}\right]\right)_{xv},
\end{equation}
which shows that
\begin{equation*}
\frac{\rho^{n+1}-\rho^n}{\tau} = \Psi.
\end{equation*}
It remains to verify that
 \begin{equation}
 \Psi= - \nabla_x \cdot j^n 
 \end{equation}
 so that \eqref{cont_d_rho} is satisfied.
Thanks to the augmentation of the basis, we can exploit equations \eqref{proj_x} and \eqref{proj_v}, getting 
\begin{equation} \label{eq:exact}
    \sum_{ijlm} X_i^{n+1} M_{ij} S_{jl}^n N^T_{lm} V_m^{n+1} =  \sum_{jl} X_j^{n}  S_{jl}^n  V_l^{n}.
\end{equation}
This allows us to derive the following identities
\begin{align}
\label{rho_2}
\Bigl(  X_k^{n+1} &,\text{RHS} \Bigl[ f = f_{0v}\sum_{ijlm} X_i^{n+1} M_{ij} S_{jl}^n N^T_{lm} V_m^{n+1}\Bigr]\Bigr)_{xv} \\ \nonumber
&= \Bigl(  X_k^{n+1} ,\text{RHS} \Bigl[ f = f_{0v}\sum_{jl} X_j^{n}  S_{jl}^n  V_l^{n}\Bigr]\Bigr)_{xv} \\ \nonumber
&= \Bigl(   X_k^{n+1} ,-v \cdot f_{0v} \nabla_x  \sum_{jl} X_j^{n} S_{jl}^n  V_l^{n}+ E^{n} \cdot \nabla_v f_{0v} \sum_{jl} X_j^{n} S_{jl}^n  V_l^{n} \Bigr)_{xv}  \\ \nonumber
&= - \sum_{jl} S_{jl}^n 
\int X_k^{n+1}  \nabla_x X_j^{n} dx \cdot \int f_{0v} v V_l^{n}\,dv\\ \nonumber
&\quad +   \sum_{jl} S_{jl}^n 
\int X_k^{n+1} E^{n}   X_j^{n} dx \cdot \int \nabla_v(f_{0v}  V_l^{n}) \,dv\\ \nonumber
&= -\sum_{jl} S_{jl}^n 
\int X_k^{n+1}   \nabla_x X_j^{n} dx \cdot \int f_{0v} v V_l^{n}\,dv. \nonumber
\end{align}
Finally, by substitution of  \eqref{rho_2} in \eqref{S1} and exploiting formula \eqref{proj_gx}, we get  
\begin{align*}
   \Psi
&= - \sum_{jl} S_{jl}^n \sum_k X_k^{n+1}
\int X_k^{n+1}  \nabla_x X_j^{n} dx \cdot \int f_{0v} v V_l^{n} \,dv\\ \nonumber
&=- \sum_{jl} S_{jl}^n \left( \sum_k X_k^{n+1}
 \langle  X_k^{n+1} ,  \nabla_x X_j^{n} \rangle_x  \right) \cdot \int f_{0v} v V_l^{n}\,dv\\ \nonumber
& =- \sum_{jl} S_{jl}^n  \nabla_x X_j^{n}    \cdot \int f_{0v} v V_l^{n}\\ \nonumber
&= -\nabla_x \cdot j^n.
\end{align*}
This proves that \eqref{cont_d_rho} is satisfied by the numerical scheme. Integrating \eqref{cont_d_rho} over $x$, we  obtain conservation of mass.
\end{proof}}

\textcolor{black}{
In order to show conservation of momentum we need, in addition to equations \eqref{proj_x}-\eqref{proj_gx} that were already used in theorem \ref{thm_1} and which were ensured by augmenting the approximation space, that no error is made in projecting $E^n \rho^n$ onto the space spanned by the $\widetilde{X}_k^{n+1}$. This condition is automatically satisfied as we demonstrate in the following lemma.
\begin{lemma} \label{lem:Erho} For all $s \in \{1, \ldots, d\}$ we have $E^n_s \rho^n \in \text{span}\{\widetilde{X}^{n+1}_k\}$. In particular,
    \begin{equation}
    \label{proj_erho}
        P_{x,n+1}(E^n \rho^n) = \sum_k  \langle \widetilde{X}_k^{n+1}, E^n \rho^n  \rangle_x  \widetilde{X}_k^{n+1} = E^n \rho^n.
    \end{equation}
\end{lemma}
\begin{proof} From the $K$ step \eqref{eq:unconv_x} we get
    \begin{align*}
        K_k^{n+1} & = K_k^{n} + \tau (V_k^n, \text{RHS}^n) =  K_k^n - \tau \sum_{j} c_{kj}^1 \cdot \nabla_x K_j^n + \tau \sum_j c_{kj}^2 \cdot E^n K_j^n.
    \end{align*}
    Clearly, both the first and second term already lie in $\text{span}\{\widetilde{X}^{n+1}_k\}$. For the third term we set $k = s+1$. This yields
    \begin{align*}
        \sum_j c^2_{s+1,j} K_j^n &= \sum_j \left(V_{s+1}, \nabla_v (f_{0v} V_j^n)\right)_v K_j^n \\
                          &= \frac{1}{\Vert v_s\Vert} \sum_j \int v_{s} \nabla_v (f_{0v}  K_j^n V_j^n) \,\mathrm{d}v \\
                          &= -\frac{e^{s}}{\Vert v_s \Vert} \int f_{0v} \sum_j K_j V_j^n \,\mathrm{d}v \\
                          &= -\frac{e^{s}}{\Vert v_s \Vert} \rho^n,
    \end{align*}
    where $e^s_i = \delta_{is}$. Therefore, $E^n_s \rho^n$ lies in $\text{span}\{\widetilde{X}^{n+1}_k\}$, which is the desired result.
\end{proof}
}

\begin{theorem}
\label{thm_2}
 Algorithm \ref{algorithm}, for $m \geq 2$, satisfies the following discrete version of the momentum continuity equation
\begin{equation}
\label{cont_d_j}
\dfrac{j^{n+1}-j^n}{\tau} + \nabla_x \cdot \sigma^n+ E^n \rho^n =0
\end{equation}
and thus preserves the total momentum.
\end{theorem}
\begin{proof}
\textcolor{black}{For notational simplicity, we conduct the proof for a one-dimensional velocity space only. We remark, however, that the theorem also holds in dimension $d$ for $m\ge d+1$.}
\textcolor{black}{Since $U_2 = v/\|v\|$ and using $K_j = \sum_i X_i S_{ij}$, the momentum density is given by
$$j = (v, f)_{v} = \|v\| (U_2 ,f)_{v} = \Vert v \Vert K_2.$$
Using a similar argument as in the proof of Theorem \ref{thm_1}, we get the following:
\begin{equation}
\label{j_1}
    \begin{aligned}
j ^ {n+1} & = \Vert v \Vert K_2^{n+1}= \Vert v \Vert\sum_{k} X_k^{n+1} S_{k2}^{n+1}  \\ 
& = \Vert v \Vert\sum_{ij} S_{ij}^n \langle V_j^n, U_2 \rangle \sum_k X_k^{n+1} \langle X_k^{n+1}, X_i^{n}  \rangle_x  +\tau \Lambda\\ 
&= \Vert v \Vert  \sum_{i} S_{i2}^n  \sum_k X_k^{n+1} \langle X_k^{n+1}, X_i^{n}  \rangle_x  + \tau \Lambda \\ 
&= \Vert v \Vert \sum_{i} S_{i2}^n  X_i^{n}   +\tau \Lambda \\ 
&= \Vert v \Vert K_2^n  +\tau \Lambda \\ 
& = j^n   +\tau  \Lambda,\\ 
    \end{aligned}
\end{equation}
where
\begin{equation}
\label{S2}
\Lambda = \Vert v \Vert \sum_k X_k^{n+1} \left(   X_k^{n+1} U_2,\text{RHS} \left[ f = f_{0v}\textstyle\sum_{ijlm} X_i^{n+1} M_{ij} S_{jl}^n N^T_{lm} V_m^{n+1}\right]\right)_{xv}.
\end{equation}
Thus
\begin{equation*}
\frac{j^{n+1}-j^n}{\tau} = \Lambda,
\end{equation*}
 and we now have to prove that 
 \begin{equation}
\Lambda= - \nabla_x \cdot \sigma^n - E^n \rho^n
 \end{equation}
in order that \eqref{cont_d_j} is satisfied.
Using \eqref{proj_erho} allows us to derive the following identities
\begin{equation}
\label{j_2}
    \begin{aligned}
\|v\|\Bigl(   &X_k^{n+1} U_2,\text{RHS} \Bigl[ f= f_{0v}\sum_{ijlm} X_i^{n+1} M_{ij} S_{jl}^n N^T_{lm} V_m^{n+1}\Bigr]\Bigr)_{xv} \\ \
& = \|v\|\Bigl(   X_k^{n+1} U_2,\text{RHS} \Bigl[ f= f_{0v}\sum_{jl} X_j^{n}  S_{jl}^n  V_l^{n}\Bigr]\Bigr)_{xv} \\ \
&=\Bigl(   X_k^{n+1} v,-v \cdot f_{0v} \nabla_x \sum_{jl} X_j^{n}  S_{jl}^n  V_l^{n}+ E^n \cdot \nabla_v f_{0v} \sum_{jl} X_j^{n}  S_{jl}^n  V_l^{n} \Bigr)_{xv}  \\ \
&= -\sum_{jl} S_{jl}^n 
\int X_k^{n+1}  \nabla_x X_j^{n} dx \cdot \int f_{0v} v^2 V_l^{n} dv\\ \
&\quad +   \sum_{jl} S_{jl}^n 
\int X_k^{n+1} E^n   X_j^{n} dx \cdot \int v \nabla_v(f_{0v}  V_l^{n}) dv \\ \
& = -\sum_{jl} S_{jl}^n 
\int X_k^{n+1}  \nabla_x X_j^{n} dx \cdot \int f_{0v} v^2 V_l^{n} dv\\ \
&\quad -   \sum_{jl} S_{jl}^n 
\int X_k^{n+1} E^n   X_j^{n} dx \cdot \int  f_{0v}  V_l^{n} dv.\\ \
\end{aligned}
\end{equation}
The last identity is obtained by integration by parts.
By substitution of \eqref{j_2} in  \eqref{S2} and using the properties \eqref{proj_x}, \eqref{proj_gx}, and \eqref{proj_erho}, we get
\begin{align*}
\Lambda
&=  - \sum_{jl} S_{jl}^n \sum_k X_k^{n+1}
\int X_k^{n+1}  \nabla_x X_j^{n} dx \int f_{0v} v^2 V_l^{n}\,dv  \\ \nonumber
&\quad -\sum_{jl} S_{jl}^n \sum_k X_k^{n+1}
\int X_k^{n+1} E^n X_j^{n} dx  \int f_{0v} V_l^{n}\,dv\\ \nonumber
&=  - \sum_{jl} S_{jl}^n  \sum_k X_k^{n+1}
 \langle  X_k^{n+1},   \nabla_x X_j^{n} \rangle_x   \int f_{0v} v^2 V_l^{n}\,dv  \\ \nonumber
&\quad - \sum_k X_k^{n+1}
\int X_k^{n+1} E^n    \int f_{0v}\sum_{jl}X_j^{n} S_{jl}^n V_l^{n}\,dv dx  \\ \nonumber
& =- \sum_{jl} S_{jl}^n  \nabla_x X_j^{n}     \int f_{0v} v^2 V_l^{n} dv -  \sum_k X_k^{n+1}\langle X_k^{n+1}, E^n \rho^n  \rangle_x \\ \nonumber
&= -\nabla_x \cdot \sigma^n- E^n \rho^n,
\end{align*}
which means that the method satisfies \eqref{cont_d_j}.
By repeating in the discrete setting the considerations of section~\ref{sec2}, for which it holds that $E^n (1-\rho^n) =  \nabla \cdot (E^n \otimes E^n-\frac{1}{2} (E^n)^2)$ and  $ \int E^n dx =0$, we get, by integrating \eqref{cont_d_j}, \textcolor{black}{the conservation of total momentum}.}
\end{proof}

\textcolor{black}{We briefly recall an alternative formulation for the expression of the energy density \eqref{e_expr}, proposed in \cite{EJ}, that will be useful for the proof of the next result. The function $U_3 = (v^2-1)/\|v^2-1\|$ can be written as
$$ v^2 = \Vert v^2 -1 \Vert U_3  + \Vert 1 \Vert U_1.   $$
Thanks to this relation, we have
\begin{align*} \int v^2 f dv  & = \sum_{ij} X_i S_{ij}\left(\Vert v^2 -1 \Vert \langle U_3, V_j \rangle_v  + \Vert 1 \Vert \langle U_1, V_j \rangle_v \right) = \Vert v^2 -1 \Vert  \sum_{i} X_i S_{i3} + \Vert 1 \Vert \sum_{i} X_i S_{i1}\\ \nonumber 
&= \Vert v^2 -1 \Vert K_3  + \Vert 1 \Vert K_1. \end{align*}
Therefore, we can write the energy density as
\begin{equation} \label{eq:ennew} 
    e = \dfrac{1}{2}\left( \Vert v^2 -1 \Vert K_3  + \Vert 1 \Vert K_1\right)  + \dfrac{1}{2}E^2.
\end{equation}
}

\textcolor{black}{
We now turn our attention to energy conservation. In order to derive the discrete version of the continuity equation below it is useful to first show the following lemma (which is similar to Lemma \ref{lem:Erho}):
\begin{lemma} It holds that $E^n \cdot j^n \in \text{span}\{\widetilde{X}^{n+1}_k\}$. In particular,
    \[ P_{x,n+1}(E^n \cdot j^n) = \sum X_k^{n+1} \langle X_k^{n+1}, E^n \cdot j^n \rangle_x =  E^n \cdot j^n. \]
\end{lemma}
\begin{proof} From the $K$ step \eqref{eq:unconv_x} we get
    \begin{align*}
        K_k^{n+1} & = K_k^{n} + \tau (V_k^n, \text{RHS}^n) =  K_k^n - \tau \sum_{j} c_{kj}^1 \cdot \nabla_x K_j^n + \tau \sum_j c_{kj}^2 \cdot E^n K_j^n.
    \end{align*}
    Clearly, both the first and second term already lie in $\text{span}\{\widetilde{X}^{n+1}_k\}$. For the third term we set $k = d+2$ to get
    \begin{align*}
        \sum_j c^2_{d+2,j} K_j^n &= \sum_j \left(V_{d+2}, \nabla_v (f_{0v} V_j^n)\right)_v K_j^n \\
                          &= \frac{1}{\Vert v^2 - 1\Vert} \sum_j \int (v^2-1) \nabla_v (f_{0v}  K_j^n V_j^n) \,\mathrm{d}v \\
                          &= -\frac{2}{\Vert v^2 - 1\Vert}  \int v f_{0v} \sum_j K_j V_j^n \,\mathrm{d}v \\
                          &= -\frac{2}{\Vert v^2 - 1\Vert}  j^n,
    \end{align*}
    Thus, $E^n \cdot j^n$ lies in $\text{span}\{\widetilde{X}^{n+1}_k\}$, which is the desired result.
\end{proof}
}

\textcolor{black}{The main result with respect to energy conservation is stated in the following theorem.} \textcolor{black}{Note that the proposed integrator does not conserve energy. However, we show that our scheme satisfies the discrete continuity equation \eqref{eq_cont_e},  which introduces an error of size $\mathcal{O}(\tau^2)$ in energy per time step, similarly to the fully explicit (but not robust) conservative Euler scheme considered in \cite{EJ}.}

\begin{theorem}\label{thm_3} Algorithm \ref{algorithm}, for $m \geq 2+d$, satisfies the following discrete version of the energy continuity equation
\begin{equation}
\label{eq_cont_e}
    \dfrac{e^{n+1} - e^n}{\tau} + \nabla_x \cdot Q^n - E^n \cdot \left(\frac{E^{n+1}-E^n}{\tau}  - j^n \right)=  \dfrac{(E^{n+1}-E^n)^2}{2 \tau}
                            =   \mathcal{O}(\tau).
\end{equation}
\end{theorem}
\begin{proof}
\textcolor{black}{According to \eqref{eq:ennew}, it holds that
\begin{equation}
\label{diff_e}
\dfrac{e^{n+1} - e^n}{\tau} = \Vert v^2 -1 \Vert \dfrac{K_3^{n+1} - K_3^n}{2 \tau} + \Vert 1 \Vert \dfrac{K_1^{n+1} - K_1^{n}}{2 \tau} +  \dfrac{(E^{n+1})^2-(E^n)^2}{2 \tau}.
\end{equation}
Since $U_3 = (v^2-1)/\|v^2-1\|$,  a similar argument as in the proofs of Theorems \ref{thm_1} and \ref{thm_2} shows
\begin{equation} \label{diff_K3}
\begin{aligned}
    \Vert v^2 -1 \Vert \dfrac{K_3^{n+1} - K_3^n}{2 \tau} &= \frac{1}{2} \sum_k X_k^{n+1} \Bigl(   X_k^{n+1} (v^2-1),\text{RHS}\Bigl[f=f_{0v}\sum_{ijlm} X_i^{n+1} M_{ij} S_{jl}^n N^T_{lm} V_m^{n+1}\Bigr]\Bigr)_{xv} \\
&=  \sum_k X_k^{n+1} \Bigl(   X_k^{n+1} \dfrac{v^2}{2},\text{RHS}\Bigl[f=f_{0v}\sum_{ijlm} X_i^{n+1} M_{ij} S_{jl}^n N^T_{lm} V_m^{n+1}\Bigr]\Bigr)_{xv} \\
& \qquad -  \frac{1}{2}\sum_k X_k^{n+1} \Bigl(   X_k^{n+1},\text{RHS}\Bigl[f=f_{0v}\sum_{ijlm} X_i^{n+1} M_{ij} S_{jl}^n N^T_{lm} V_m^{n+1}\Bigr]\Bigr)_{xv}. 
\end{aligned}
\end{equation}
Using \eqref{eq:exact} we derive
\begin{equation}
\begin{aligned}
\label{K3}
\Bigl(   X_k^{n+1} & \frac{v^2}{2},\text{RHS}\Bigl[f=f_{0v}\sum_{ijlm} X_i^{n+1} M_{ij} S_{jl}^n N^T_{lm} V_m^{n+1}\Bigr]\Bigr)_{xv} \\ 
&= \Bigl(   X_k^{n+1} \frac{v^2}{2},\text{RHS}\Bigl[f=f_{0v}\sum_{jl} X_j^{n}  S_{jl}^n  V_l^{n}\Bigr]\Bigr)_{xv} \\ 
&= \Bigl(   X_k^{n+1} \frac{v^2}{2},-v \cdot f_{0v} \nabla_x  \sum_{jl}  X_j^{n}  S_{jl}^n  V_l^{n}+ E^n \cdot \nabla_v f_{0v} \sum_{jl}  X_j^{n}  S_{jl}^n  V_l^{n} \Bigr)_{xv}  \\ 
&= -\sum_{jl} S_{jl}^n 
\int X_k^{n+1}  \nabla_x X_j^{n} dx \cdot \int \frac{v^2}{2} v  f_{0v}  V_l^{n}\,dv\\ 
&\quad +   \sum_{jl} S_{jl}^n 
\int X_k^{n+1} E^n   X_j^{n} dx \cdot \int \frac{v^2}{2}\nabla_v(f_{0v}  V_l^{n}) \,dv \\ 
& = -\sum_{jl} S_{jl}^n 
\int X_k^{n+1}  \nabla_x X_j^{n} dx \cdot \int \frac{v^2}{2} v  f_{0v}  V_l^{n}\,dv\\ 
&\quad -   \sum_{jl} S_{jl}^n 
\int X_k^{n+1} E^n   X_j^{n} dx \cdot \int v  f_{0v}  V_l^{n} dv.   
\end{aligned}
\end{equation}
By substitution of \eqref{K3} and according to the definition of $Q$ in section \ref{sec2}, we can write the following quantity as
\begin{align*}
\sum_k X_k^{n+1} &\Bigl(   X_k^{n+1} \frac{v^2}{2},\text{RHS}\Bigl[f=f_{0v}\sum_{ijlm} X_i^{n+1} M_{ij} S_{jl}^n N^T_{lm} V_m^{n+1}\Bigr]\Bigr)_{xv} \\ \nonumber
&=-\sum_{jl} S_{jl}^n
  \nabla_x X_j^{n}  \cdot  \int \frac{v^2}{2} v \cdot f_{0v}  V_l^{n} \,dv
  - \sum_k X_k^{n+1} \int X_k^{n+1}\sum_{jl} S_{jl}^n
  E^n X_j^{n} dx \cdot  \int v \cdot f_{0v}  V_l^{n} \,dv \\ \nonumber
  & = - \nabla_x   \cdot  \int \frac{v^2}{2} v  f_{0v} \sum_{il} X_i^{n} S_{il}^n  V_l^{n} dv -\sum_k X_k^{n+1} \int X_k^{n+1} E^n \cdot  \int  v f_{0v} \sum_{il} X_i^{n} S_{il}^n  V_l^{n} \,dv dx \\ \nonumber
 &= -\nabla_x \cdot Q^n - \sum_k X_k^{n+1} \int X_k^{n+1} E^n j^n dx\\ \nonumber
  &= -  \nabla_x \cdot Q^n -  E^n \cdot j^n.
\end{align*}
From the proof of Theorem \ref{thm_1}, we know that
\begin{align}
\label{K3_2}
\sum_k X_k^{n+1} &\Bigl(   X_k^{n+1},\text{RHS}\Bigl[f=f_{0v}\sum_{ijlm} X_i^{n+1} M_{ij} S_{jl}^n N^T_{lm} V_m^{n+1}\Bigr]\Bigr)_{xv} = -\nabla_x \cdot j^n.
\end{align}
Inserting \eqref{K3} and \eqref{K3_2} into \eqref{diff_K3}, we get
\begin{align}
\label{diff_K3_1}
\Vert v^2 -1 \Vert \dfrac{K_3^{n+1} - K_3^n}{2 \tau}
= -\nabla_x \cdot Q^n -  E^n \cdot j^n + \frac{1}{2}\nabla_x \cdot j^n.
\end{align}
From  Theorem \ref{thm_1} we know that
\begin{equation}
\label{diff_K1}
\|1\| \dfrac{K_1^{n+1}-K_1^n}{2 \tau} = - \frac{1}{2} \nabla_x \cdot j^n.
\end{equation}
Inserting \eqref{diff_K3_1} and \eqref{diff_K1} into \eqref{diff_e}, we obtain
\begin{align*}
\dfrac{e^{n+1}-e^{n}}{\tau} & = -\nabla_x \cdot Q^n -  E^n \cdot j^n + \dfrac{(E^{n+1})^2-(E^n)^2}{2 \tau} \nonumber \\
& =  -\nabla_x \cdot Q^n -  E^n \cdot j^n + \dfrac{(E^{n+1}-E^n)^2 + 2 E^n E^{n+1} - 2 (E^n)^2}{2 \tau} \nonumber \\
& = - \nabla_x \cdot Q^n + E^n \cdot \left(\frac{E^{n+1}-E^n}{\tau}  - j^n \right) + \dfrac{(E^{n+1}-E^n)^2}{2 \tau}.
\end{align*}
Since $E^{n+1}$ is always $\mathcal{O}(\tau)$ close to $E^n$, the proof is complete.
} 
\end{proof}

\subsection{Space discretization \& implementation}

In the previous section we have shown that the proposed integrator conserves mass and momentum exactly and energy up to a time discretization error. We have not dealt with the question of space discretization and the actual implementation in finite precision arithmetics.

It is clear that in order to obtain a conservative scheme the space discretization has to satisfy certain properties of the continuous derivatives that we used in the proof. In particular, we require:
\begin{enumerate}
    \item[(i)] The space discretization has to mimic (replace integration by summation and differentiation by the discrete approximation used) the equation
        \[ \int \nabla_x u  \,dx = 0 \]
        for arbitrary $u$. This is naturally satisfied for centered difference schemes and Fourier spectral methods,  among others.
     \item[(ii)] To prove Theorem 2, i.e. the conservation of momentum, we used integration by parts
    $$ \int v \nabla_v(f_{0v} V^{n+1}_l) \,dv = -\int f_{0v} V^{n+1}_l \,dv.$$
        This property has to be mimicked at the discrete level (i.e.~replacing integrals by sums and derivatives by the discrete approximation used). This has to be taken into account when implementing the scheme.
    In particular, for centered differences we should not rely on the product rule. In fact, for centered differences the discretization of $\nabla_v(f_{0v} V)$ is different from the discretization of $\nabla_v(f_{0v}) V + f_{0v} \nabla_v V$, i.e. the product rule is not exactly satisfied. This is also crucial in the computation of the coefficient $c^2_{kj}$.
\end{enumerate}
Those observations are also important to improve the conservation of energy, since a similar argument of integration by parts is present in the proof of Theorem 3.

In order to get conservation up to machine precision we have to be careful when implementing the following routines in finite precision. In particular, in our implementation we perform the following steps.
\begin{enumerate}
\item[(i)] To avoid incorrect non-zero values in the discrete computation of $c^2_{ak}$, with $a \in \lbrace1, .., m\rbrace$ (i.e. $a$ is an index associated to fixed functions) we directly insert the exact values for the fixed basis functions.
\item[(ii)]Even if the fixed functions $U_1$, $U_2$ and $U_3$ are analytically orthonormal, due to the truncated velocity domain this is not necessarily true up to machine precision in the implementation. Thus, we explicitly orthonormalize the $U_a$ before commencing each time step of the algorithm.
\end{enumerate}

\section{Rank adaptivity\label{rank_adaptivity}}

In this section we will describe an approach to adaptively choose the rank in the time integration. This is important because for a practitioner choosing a fixed rank that satisfies a given accuracy constraint can be difficult. In addition, choosing a fixed rank might be suboptimal in the sense that the rank required at a given time might vary significantly. If this is the case then taking a fixed rank would imply that we always have to choose the largest rank as a function of time.

In \cite{ceruti2022adaptive} and \cite{schrammer_adaptive} some strategies to adaptively choose the rank for dynamical low-rank integrators have been proposed. The main idea is to, if necessary, reduce the rank by removing singular values of $S$ that are below are certain tolerance $\theta$. If the rank needs to be enlarged the dynamical low-rank algorithm is used to propose new basis functions appropriate for the problem at hand. The adaptive low-rank integrator retains the exactness, robustness and symmetry-preserving properties, see \cite{ceruti2022adaptive}.  Similar approaches to rank adaptivity have been used in \cite{kusch2021robust}. A generalization to tensors is also available \cite{ceruti2022badaptive}.

In the proposed algorithm (Algorithm \ref{algorithm}) we have already added basis functions before performing the low-rank algorithm. This is required for the algorithm in order to be conservative (see section \ref{sec:robustcons}).
Thus, we will always be in a situation where we need to truncate the rank after a step of the unconventional integrator has been performed. Following \cite{ceruti2022adaptive} this can be accomplished as follows.
\begin{enumerate}
    \item[1.] \textit{Truncation with respect to the error in the solution}: from algorithm \ref{algorithm} we get $\hat{S}$ (with dimension $\hat{r}$) by performing a singular value decomposition of $\widetilde{S}^{rem}$. If $\hat{\sigma}_k $, for $k = 1,\ldots, \hat{r}$ are the singular values of $\hat{S}$, we choose the new rank $r$ ($r \leq \hat{r}$) such that
$$ \hat{\sigma}^2_{r+1}+\ldots+ \hat{\sigma}^2_{\hat{r}} \leq \theta^2.$$
\end{enumerate}
In this case the tolerance $\theta$ controls the error made by the low-rank approximation in the density function $f(t,x,v)$.

In many applications, however, it might not be the best choice to control the error of the density function itself. Often, practitioners are more interested in averaged quantities such as the electric energy. Thus, we take the viewpoint here that it is advantageous to perform the truncation by taking into account what physical quantities are of ultimate interest. Often, a smaller rank can be sufficient to, e.g.,~approximate the electric energy up to a certain tolerance, than the rank that would be required to obtain the density function up to the same level of tolerance. Such an approach can be easily incorporated within the dynamical low-rank integrator considered here. For this strategy we proceed as follows.
\begin{enumerate}
    \item[2.] \textit{Truncation with respect to error in the electric energy}: if $e$ is the electric energy computed at this stage (i.e.~without performing a truncation), we decrease the rank until the error with respect to the electric energy is bounded by $\theta$. More specifically, we denote by $e_{l}$ the electric energy computed by only using the first $l$ singular values. The new rank $l$ is then chosen as follows
        $$l \text{\; is the minimum such that \;} \vert e-e_{l}\vert \leq \theta. $$
        If the rank $l$ becomes less than $m$, the number of fixed functions, this means that the local error in the electric energy is always below the prescribed threshold $\theta$. In this situation, we set the new rank to $\bar{r}$, where $\bar{r}>m$ is fixed a priori.
\end{enumerate}

Another option would be to prescribe a certain tolerance for the error in energy (which is not exactly conserved by our algorithm). In this case we proceed as follows.
\begin{enumerate}
    \item[3.] \textit{Truncation with respect to the error in energy}: this procedure is similar to strategy $2$, apart from the target quantity. We choose the new rank $l$ as follows
        $$l \text{\; is the minimum, such that \;} \vert \mathcal{E}_{l} - \mathcal{E}^n \vert \leq \theta, $$
        where $\mathcal{E}^n$ is the value for the energy at the previous time, i.e. computed at $t^n$, and $\mathcal{E}_l$ is value of the energy at the current time instant $t^{n+1}$ with rank $l$.
\end{enumerate}

\section{Numerical results}
\label{numerics}

In this section we will provide a number of numerical results to verify both the conservative properties of our scheme as well as the rank adaptive algorithm. We consider the classic test problems of linear and nonlinear Landau damping as well as a two-stream instability. The most pronounced improvement in the qualitative features of the solution are observed for nonlinear Landau damping, where the conservative algorithm can produce a numerical solution that matches well with the exact solution even though at the same rank the classic dynamical low-rank algorithm provides incorrect results. We further consider, for the two-stream instability, the effect of using rank adaptivity with respect to different goal functions (error in the density function $f$, error in the electric energy, and error in total energy) and observe that a significantly smaller rank is often sufficient to approximate certain physically important quantities such as the electric energy (as opposed to the rank that would be required to approximate the density function  up to a given accuracy). In all experiments we use $f_{0v} = \exp(-v^2/2)$ as weight function.

\subsection{Linear Landau damping \label{sec:ll}}
First, we consider the linear Landau damping problem with the rank $1$ initial value
\begin{equation*}
f(0,x,v) = \left( 1 + \alpha \cos(k x)\right)  \dfrac{e^{-v^2/2}}{\sqrt{2 \pi}}.
\end{equation*}
The parameters are set to $\alpha =10^{-2}$ and  $k =0.5$ and the computational domain is $(x, v) \in \left[  0, 4 \pi\right] \times  \left[-6, 6 \right]$.
All simulations are done with a time step size $\Delta t = 10^{-3}$ and $128$ grid points in both the spatial and velocity direction are used. We present results with respect to all the configurations introduced in the previous sections. In particular, we consider $m=0$, $m=1$ with $U_1 \propto 1$, $m=2$ with $U_1 \propto 1$ and $U_2 \propto v$, and, finally, $m=3$ with $U_1 \propto 1$, $U_2 \propto v$ and $U_3 \propto v^2-1$.
We compute the relative error for the mass and the absolute error for the momentum.
In Figure \ref{figure_1} (top-left), we show the decay of the electric energy for $m=0$ (no conservation), $m=1$ (mass conservation),  $m=2$ (momentum conservation), and $m=3$ (energy conservation). We remark that for the linear Landau damping problem all configurations with rank $r=10$ give very good results. In particular, the analytic decay rate is accurately reproduced and the different solutions match each other very well. For $m\geq 1$ we observe a reduction in the error in mass from somewhat above $10^{-9}$ (which is already very accurate) to slightly below $10^{-11}$. Moreover, the error in momentum can be reduced from approximately $10^{-6}$ to $10^{-10}$, i.e. by four orders of magnitude. Thus, for $m\geq 2$ both mass and momentum are conserved up to machine precision.

In Figure \ref{figure_1} we also investigate conservation of energy. Note that in this case not only the error due to the truncation of the low-rank approximation needs to be taken into account, but also the fact that we use an explicit time integrator that does not conserve energy. Nevertheless, the proposed method significantly improves energy conservation (from approximately an error of $10^{-5}$ to $4 \cdot 10^{-7}$).

\begin{figure}[H]
  \includegraphics[width=8cm]{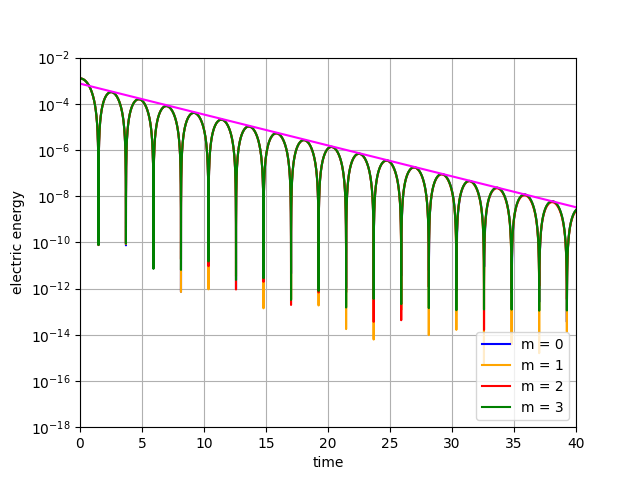}
  \includegraphics[width=8cm]{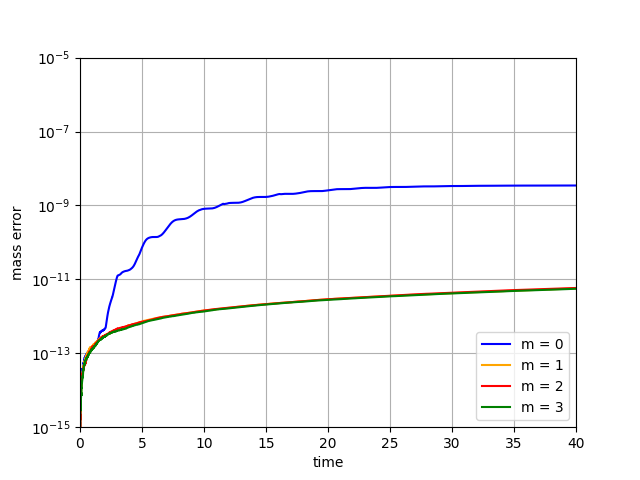}
  \includegraphics[width=8cm]{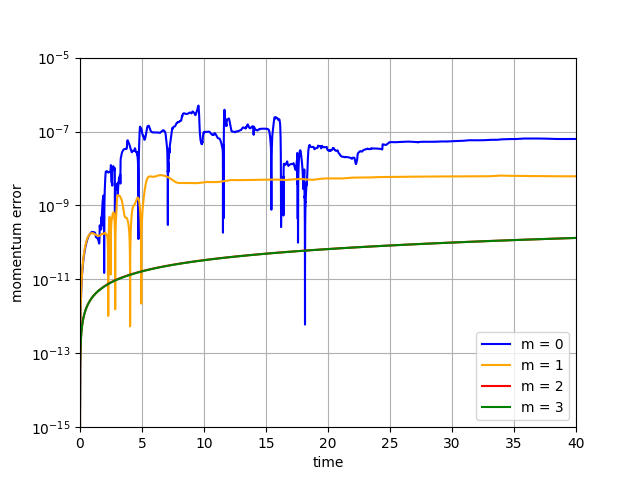}
  \includegraphics[width=8cm]{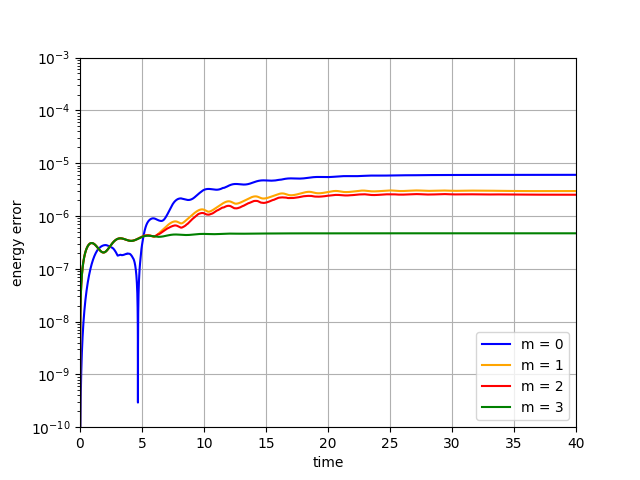}
    \caption{Evolution of the electric energy (top-left), error in mass (top-right), the error in momentum (bottom-left), and error in energy (bottom-right) for linear Landau damping. The analytically derived decay rate is shown in magenta.  Note that, e.g.,~the yellow, orange and red curves overlap in the plot for mass.  The chosen rank is $r=10$  and the configurations considered  are $m=0$ , $m=1$, $m=2$ and $m=3$. }\label{figure_1}
\end{figure}

\subsection{Nonlinear Landau damping}
In this section, we consider some numerical experiments for the nonlinear Landau damping problem. The initial value
\begin{equation*}
f(0,x,v) = \left( 1 + \alpha \cos(k x)\right)  \dfrac{e^{-v^2/2}}{\sqrt{2 \pi}}
\end{equation*}
and the computational domain $(x, v) \in \left[  0, 4 \pi\right] \times  \left[-6, 6 \right]$ are the same as in section \ref{sec:ll}. We now, however, use $\alpha = 0.5$ and  $k =0.5$. This creates strong nonlinear effects that, after an initial decay in the electric energy, result in an increase in electric energy.
As in the previous example, we discretize space and velocity with $128$ grid points and use a time step size of $\Delta t = 10^{-3}$. In these experiments, the rank used is $r=25$.

The numerical results in Figure \ref{figure_2} show that for the classic dynamical low-rank algorithm (i.e. $m=0$) the computed solution does not match the expected behavior. In particular, the expected increase in the electric energy is not observed. Using the conservative numerical scheme, on the other hand, the solution matches the expected behavior very well. Thus, in this example, using the conservative scheme markedly improves the qualitative behavior of the solution.

In addition, Figure \ref{figure_2} shows that we obtain mass and momentum conservation up to machine precision for $m \geq 1$ and $m \geq 2$, respectively, improving the error made in these quantities compared to the classic dynamical low-rank algorithms by approximately $10$ orders of magnitude. Moreover, energy conservation is significantly improved for $m=3$, by approximately 2 orders of magnitude.

\begin{figure}[H]
     \includegraphics[width=8cm]{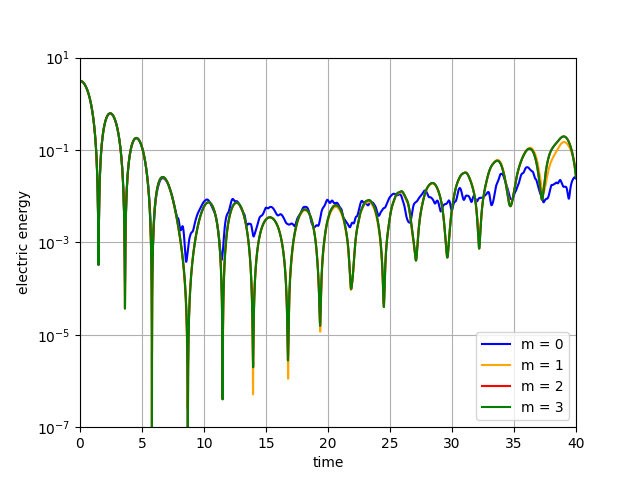}
     \includegraphics[width=8cm]{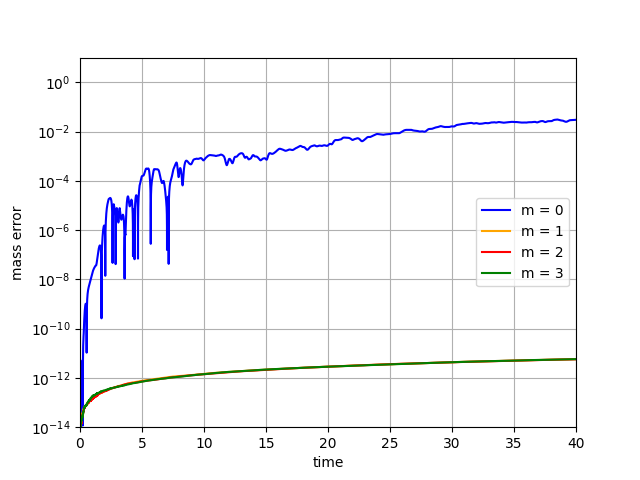}
     \includegraphics[width=8cm]{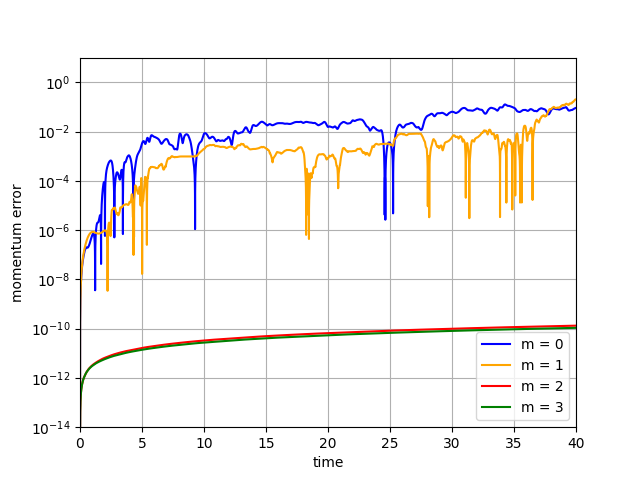}
     \includegraphics[width=8cm]{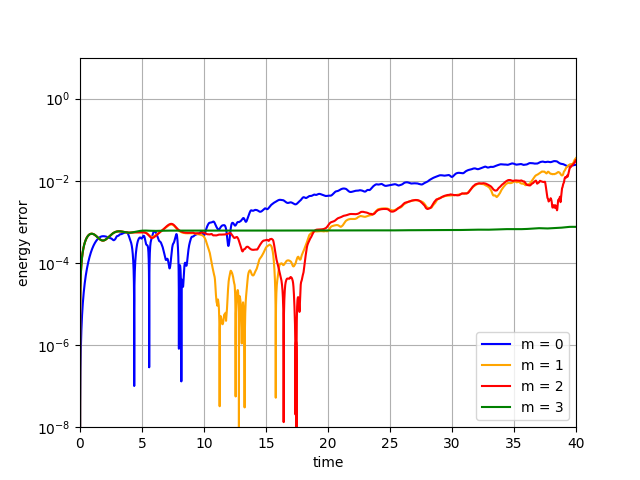}
    \caption{Evolution of the electric energy (top-left), error in mass (top-right), the error in momentum  (bottom-left), and error in energy (bottom-right) for nonlinear Landau damping.
Note that, e.g.,~in the plot for mass conservation the yellow, orange and red curves overlap.  The chosen rank is $r=25$  and the configurations considered  are $m=0$ , $m=1$, $m=2$ and $m=3$. }\label{figure_2}
\end{figure}

\subsection{Two stream instability}
We consider the  two stream instability
\begin{equation*}
f(0,x,v) = \dfrac{1}{2 \sqrt{2 \pi}}\left( e^{-(v-\bar{v})^2/2}+e^{-(v+\bar{v})^2/2} \right) \left( 1 + \alpha\cos(k x )\right)
\end{equation*}
with $\alpha =10^{-3}$, $k =0.2$, and $
\bar{v}=2.4$ on the domain $(x, v) \in \left[  0, 10 \pi\right] \times  \left[-7, 7 \right]$. We choose again $128$ grid points in both the spatial and the velocity direction and $\Delta t = 10^{-3}$ in all the experiments. We report the relative mass error and the absolute momentum error (since the total momentum in the system is zero).

In Figure \ref{figure_3}, we compare the error in mass and momentum for $m=0$, $m=1$, $m=2$ and $m=3$. The considered rank is  $r=10$. It is well known that the two-stream instability is a difficult problem for low-rank approximations with respect to conservation. As the instability progresses, the error in mass, momentum, and energy for the classic scheme (i.e.~$m=0$) increases until it is close to $\mathcal{O}(1)$. At that point the physical interpretation of the solutions seems problematic. As can be observed in Figure \ref{figure_3}, the error in mass for $m\geq 1$ and the error in the momentum for $m \geq 2$ are reduced to machine precision, as expected. Moreover, the conservative scheme is able to reduce the error in energy from $10^{-2}$ ($m=0$) to $10^{-5}$ ($m=3)$, i.e. by three orders of magnitude.

\begin{figure}[H]
     \includegraphics[width=8cm]{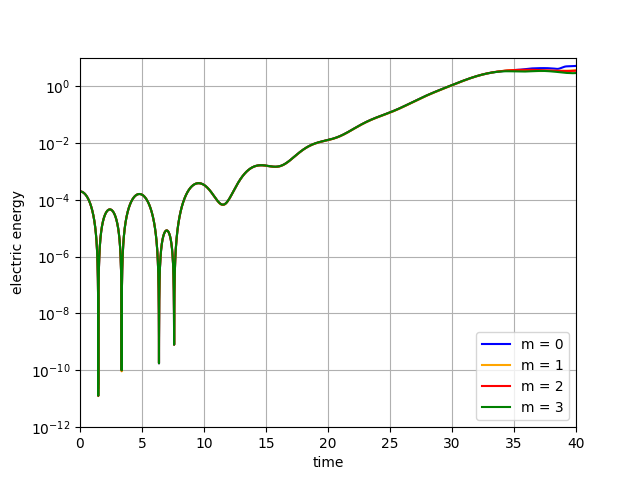}
     \includegraphics[width=8cm]{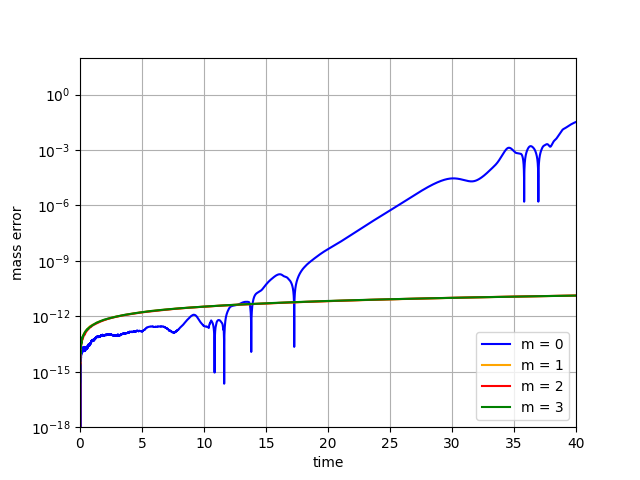}
     \includegraphics[width=8cm]{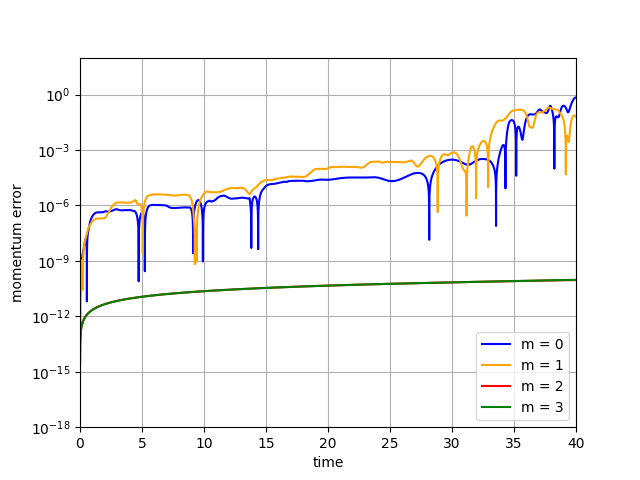}
     \includegraphics[width=8cm]{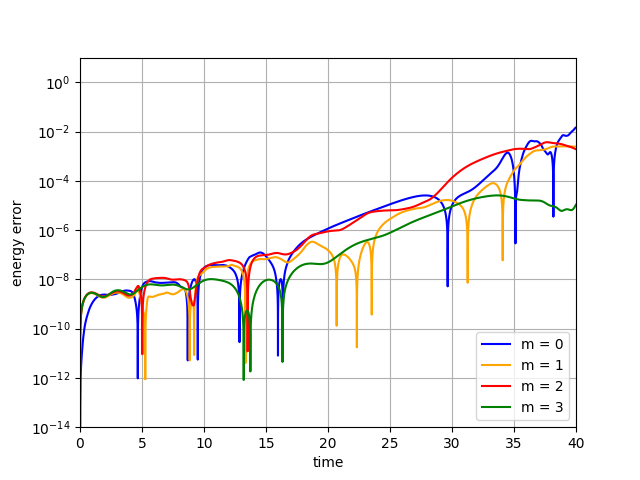}
    \caption{Evolution of the electric energy (top-left), error in mass (top-right), error in momentum (bottom-left) and error in energy (bottom-right) for the two stream instability.
Note that, e.g.,~in the mass conservation plot the yellow, orange and red curves overlap. The rank used is $r=10$.  }\label{figure_3}
\end{figure}

Finally, let us consider rank adaptivity. Most commonly such an approach uses the singular values of $S$ in order to control the error in the density function $f$. This is illustrated in Figure \ref{figure_6}, where we see that the rank of the simulation increases significantly as we enter into the nonlinear part of the dynamics.

\begin{figure}[H]
    \includegraphics[width=8cm]{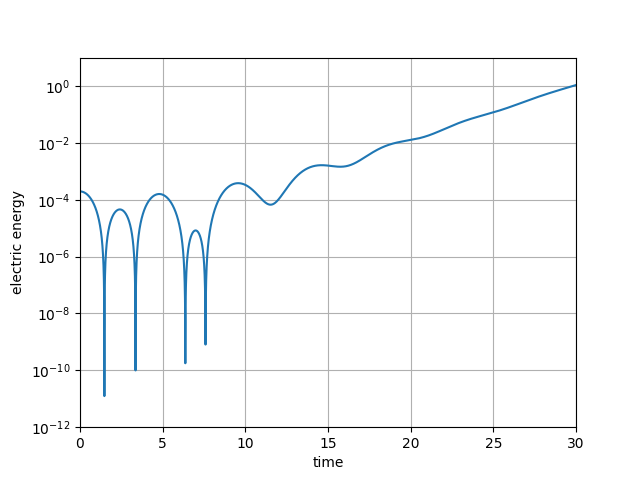}
    \includegraphics[width=8cm]{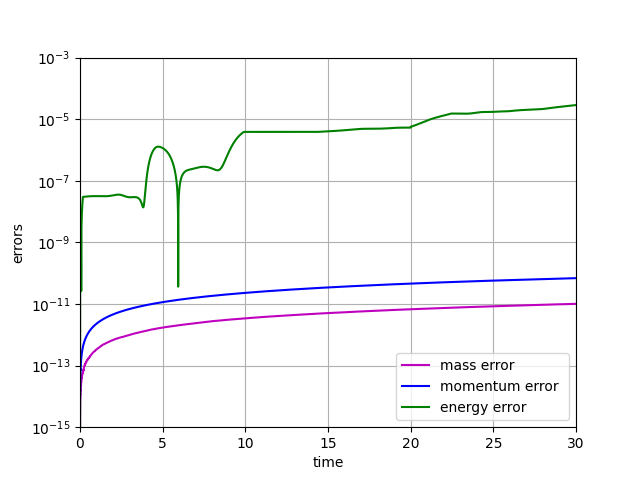}
    \includegraphics[width=8cm]{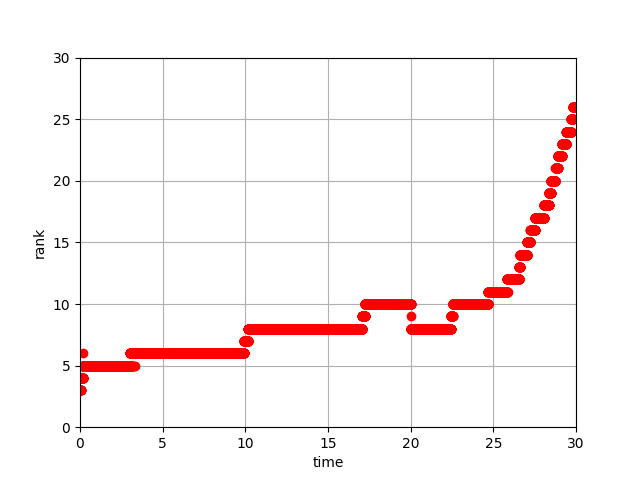}
    \caption{Rank adaptive scheme for the two stream instability with the error in the solution as the metric of interest. On the top-left, the corresponding time behavior of the electric energy. On the top-right, plot of mass, momentum and energy errors for $m=2$. On the bottom-left, the behavior of the rank as a function of time, with tolerance $\theta =10^{-9}$ until $t=20$, then $\theta = 10^{-7}$.
    }\label{figure_6}
    \end{figure}

An interesting aspect in that regard is that we do not necessarily have to adapt the rank according to the error in the density function. In many cases the quantity of interest is the electric energy. In such a situation we can use the tolerance to control the error in the electric energy. The corresponding numerical results are shown in Figure \ref{figure_9}. In fact, we can see here that rank $r=10$ is sufficient to resolve the electric energy up to an error per unit time step of size $\theta = 10^{-10}$. Thus, a significantly lower rank is sufficient in order to obtain accurate results in terms of the electric energy as opposed to the density function.

\begin{figure}[H]
    \includegraphics[width=8cm]{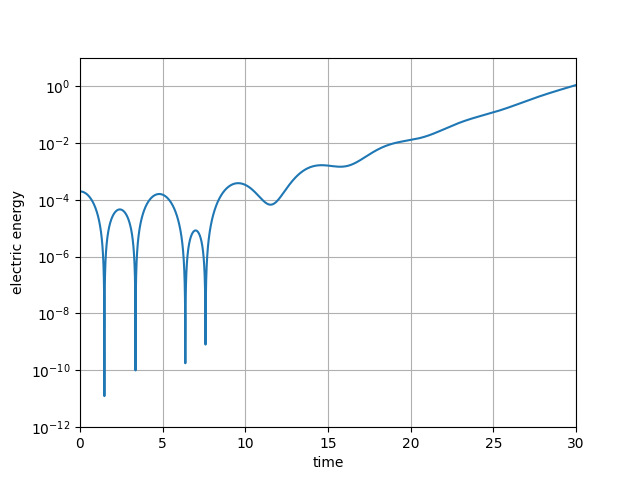}
    \includegraphics[width=8cm]{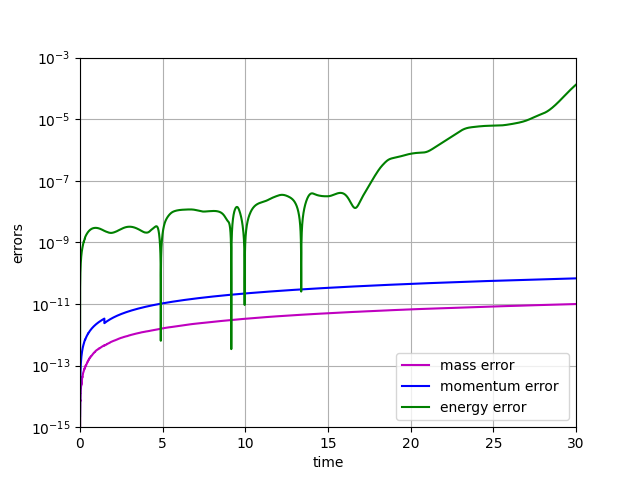}
  \includegraphics[width=8cm]{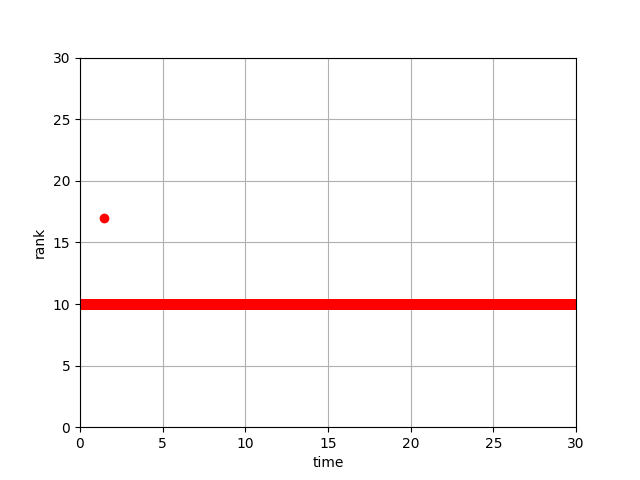}
    \caption{Rank adaptive scheme for the two stream instability with the error in the electric energy as the metric of interest. On the top-left, the corresponding time behavior of the electric energy. On the top-right, plot of mass, momentum and energy errors for $m=2$. On the bottom-left, the behavior of the rank as a function of time, with tolerance $\theta = 10^{-10}$. The minimum rank is $r = 10$.}\label{figure_9}
\end{figure}

Another possible goal function for the adaptive algorithm is the error in energy. Such an approach enforces that the error in energy per unit time step is below the tolerance. The corresponding numerical results are shown in Figure \ref{figure_7}. The tolerance chosen here is very small ($\theta=10^{-12}$) in order to illustrate the adaptive algorithm.

\textcolor{black}{In our view adapting the rank based on macroscopic quantities (such as the error in electric energy or the error in energy) can be very useful for practical simulations. As the results show a significantly smaller rank can often be used compared to using the error in the distribution function as the figure of merit. From a physical point of view we are often primarily interested in macroscopic quantities (such as the electric energy) in any case and thus accurate results can be obtained at significantly reduced computational cost.}

\begin{figure}[H]
{\includegraphics[width=8cm]{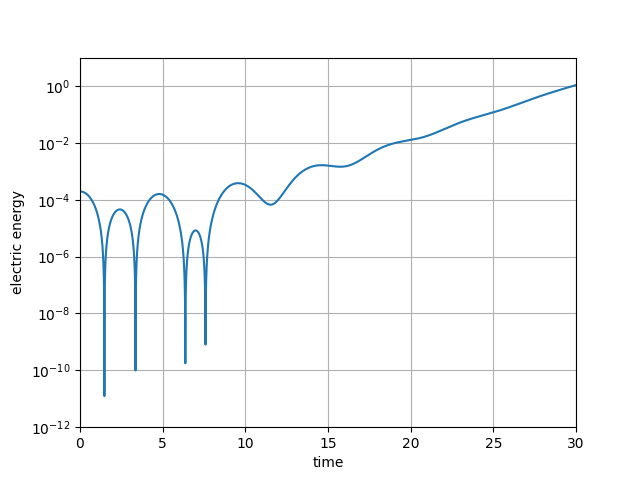}}
  {\includegraphics[width=8cm]{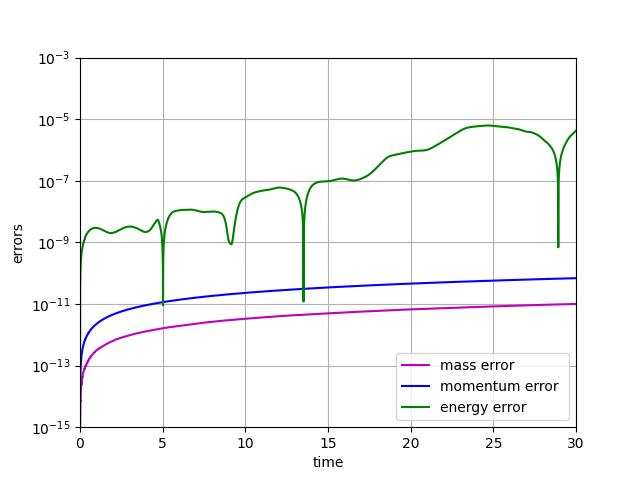}}
{\includegraphics[width=8cm]{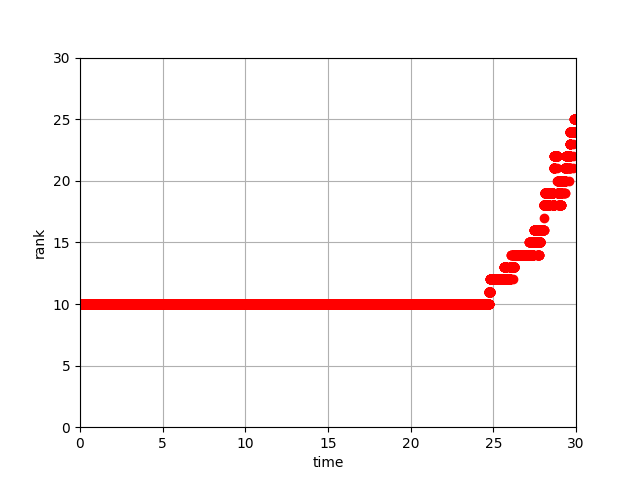}}
    \caption{Rank adaptive scheme for the two stream instability with the error in the energy as the metric of interest. On the top-left, the corresponding time behavior of the electric energy.  On the top-right, the  errors in mass, momentum and energy evolving in time for $m=2$. On the bottom-left, the behavior of the rank as a function of time, with $\theta = 10^{-12}$.}\label{figure_7}
\end{figure}

\section*{Acknowledgements}
C. Scalone is supported by GNCS-INDAM project and  PRIN2017-MIUR project 2017JYCLSF ``Structure preserving approximation of evolutionary problems``.

\bibliographystyle{plain}
\bibliography{references}

\end{document}